\documentclass{amsart}
\usepackage{amsmath,amscd}
\usepackage{amsfonts}
\usepackage{amssymb}
\usepackage{graphicx}
\usepackage{amsthm}
\usepackage{newlfont}

 \newtheorem{thm}{Theorem}[section]
 \newtheorem{cor}[thm]{Corollary}
 \newtheorem{lem}[thm]{Lemma}
 \newtheorem{prop}[thm]{Proposition}

\theoremstyle{remark}
 \newtheorem{rem}[thm]{Remark}

 \numberwithin{equation}{section}
\begin{document}

\title[Considerations  on the subgroup commutativity degree...]
 {Considerations on the subgroup commutativity degree and related notions}

%\thanks{}

%\author[R. K. Nath]{Rajat Kanti Nath}
%\address{Departement of Basic Sciences\\
%Assam Don Bosco University\\
%Guwahati-17, Assam, India} \email{rajatkantinath@yahoo.com}

\author[F. G. Russo]{Francesco G. Russo}
\address{Dipartimento di Matematica e Informatica, Universit\'a di Palermo, Via Archirafi 34, 90123, Palermo, Italy.
} \email{francescog.russo@yahoo.com}

%----------classification, keywords, date


\date{\today}
%----------additions
%\dedicatory{}
%%% ----------------------------------------------------------------------

\begin{abstract}
The concept of subgroup commutativity degree of a finite group $G$ is arising interest in several areas of group theory in the last years, since it  gives a measure of the probability that a randomly picked pair $(H,K)$ of subgroups of $G$ satisfies the condition $HK=KH$.  In this paper, a stronger notion is studied and relations with the commutativity degree are found.     
\end{abstract}

%%% ----------------------------------------------------------------------
\maketitle
%%% ----------------------------------------------------------------------

\section{Introduction}
In the present paper we deal only with finite group, even if there is a recent interest to the subject in the context of infinite groups \cite{aor, erfanian-rezaei, er, hr, rezaei-russo3}.  The \textit{commutativity degree} of  a group  $G$, given by
\begin{equation}d(G) = \frac{|\{(x,y) \in G \times G \ | \ [x,y] = 1\}|}{|G|^2}=\frac{1}{|G|^2} \sum_{x \in G} |\{y \in G  \ | \ y^{-1}xy = x\}|\end{equation}
\[ = \frac{1}{|G|^2} \sum_{x \in G} |C_G(x)|,\]
was studied extensively in \cite{ar, bM06, cms,  elr, err, guralnick, lescot1, lescot2, salemkar, nr1, nr2, nrr, rusin} an generalized in various ways. Its importance is testified in the theory of the groups of prime power orders in \cite[Chapter 2]{b1}, where it is called \textit{measure of commutativity} by Y. Berkovich in order to emphasize the fact that it really gives a measure of how far is the group from being abelian. In \cite{das-nath1,das-nath2,elr} it was introduced the following variation,
\begin{equation}d(H,K) = \frac{|\{(h,k) \in H \times K \ | \ [h,k] = 1\}|}{|H||K|} = \frac{1}{|H||K|} \underset{h \in H} {\sum} |C_K(h)|,\end{equation}
where $H$ and $K$ are two arbitrary subgroups of $G$. Of course,  $d(G,G) = d(G)$, whenever $H=K=G$, and, consequently, the bounds known in literature for $d(G)$ may be sharpened by  examining $d(H,K)$.  In recent years, there is an increasing interest in studying the problem from the point of view of the lattice theory (see \cite{f1,f2,f3,russo1,russo2}). 
T\v{a}rn\v{a}uceanu \cite{mt, mtbis} has introduced the  \textit{subgroup commutativity degree} of a finite group, that is, the ratio
 \begin{equation}sd(G) = \frac{|\{(H,K) \in \mathcal{L}(G) \times \mathcal{L}(G) \ | \ HK = KH\}|}{|\mathcal{L}(G)|^2},\end{equation}
where $\mathcal{L}(G)$ denote the subgroup lattice of $G$. It turns out that 
\begin{equation}sd(G) = \frac{1}{|\mathcal{L}(G)|^2} \underset{H \in \mathcal{L}(G)}{\sum}|\mathcal{C}(H)|,\end{equation}
where \begin{equation}\mathcal{C}(H) = \{K \in \mathcal{L}(G) \ | \ HK =  KH\}.\end{equation} Variations on this theme have been considered in \cite{barman, f1, f2, f3, or, russo1, russo2}, involving weaker notions of permutability among subgroups. 

Of course, if $[H, K] = 1$, then $HK = KH$, where $[H, K] = \langle [h, k] \ | \ h \in H, k \in K  \rangle$. Conversely, $HK = KH$ does not imply that $[H, K] = 1$. In fact, the equality among the sets $\{hk \ | \ h \in H, k \in K\}$ and $\{kh  \ | \ k\in K, h \in H\}$ does not imply, in general, that all the elements of $H$ permute with all elements of $K$. Many examples can be given. Therefore it is meaningful to  define the following ratio
\begin{equation}
ssd(G) = \frac{|\{(H, K) \in \mathcal{L}(G)^2 \ | \ [H,K] = 1\}|}{|\mathcal{L}(G)|^2},
\end{equation}
which we will call \textit{strong subgroup commutativity degree} of $G$. It is easy to see that 
\begin{equation}
ssd(G) = \frac{1}{|\mathcal{L}(G)|^2} \underset{H \in \mathcal{L}(G)}{\sum}|Comm_G(H)|,
\end{equation}
where \begin{equation}Comm_G(H) = \{K \in \mathcal{L}(G) \ | \ [H,K] = 1\},\end{equation} and that $ssd(G)$ is the probability that the subgroup $[H, K]$ of an arbitrarily chosen pair of subgroups $H, K$ of a  group $G$ is equal to the trivial subgroup of $G$. Equivalently, $ssd(G)$ expresses the probability that, randomly picked two subgroups of $G$, the subgroup generated by their commutators is trivial, and, in particular, the two subgroups are permutable. The present paper is devoted to study this notion, showing that it is related  to the previous investigations in the area of the measure theory of finite groups.

\section{Some basic properties}
There are some considerations which come by default with the strong subgroup commutativity degree. A group $G$ is \textit{quasihamiltonian}, if all pairs of its subgroups are permutable. $G$  is called \textit{modular}, if $\mathcal{L}(G)$ satisfies the well--known \textit{modular law} (see \cite{rs}). Quasihamiltonian groups were classified by  Iwasawa (see \cite[Chapter 6]{b1} or \cite[Chapter 2]{rs}), who proved that they are nilpotent and modular. This is equivalent to say that a group $G$ is quasihamiltonian if and only if all its Sylow $p$-subgroups are modular (see \cite[Exercise 3 at p.87]{rs}), being $p$ a prime.  Therefore the knowledge of quasihamiltonian groups may be reduced to that of modular $p$-groups. In literature,  for $m\ge3$ the groups \begin{equation}M(p^m)=\langle x,y \ | \ x^{p^{m-1}}=y^p=1, y^{-1}xy=x^{p^{m-2}+1}\rangle= \langle y \rangle \ltimes \langle x \rangle, \end{equation} are nonabelian modular $p$-groups and their properties have interested the researches of many authors in various contexts (see \cite{b1, rs, mt}). An immediate observation is the following. If $G=M(p^m)$, then $[\langle x \rangle, \langle y \rangle]\not=1$ and consequently $sd(G)=1$ but $ssd(G)\not=1$. In this sense, it is important to know when the strong subgroup commutativity degree is trivial.

\begin{prop}\label{trivial} 
A group $G$ has $ssd(G)=1$ if and only if it is abelian.
\end{prop}

\begin{proof} We have that
$ssd(G) = 1$ if, and only if, $[H, K] = 1$ for all subgroups $H$ and $K$ of $G$, if, and only
if, $[h, k] = 1$ for all $h \in H$, $k \in K$ and for all $H$ and $K$ in $G$. This implies, in
particular, that $[h, k] = 1$ for all $h, k \in G$, that is, $G$ is abelian. Conversely,
if $G$ is abelian, then it is clear that $ssd(G) = 1$.
\end{proof}

On another hand, the following relation shows that $ssd(G)$ is related to $d(H,K)$ in a deep way.

\begin{thm}\label{t:fundamental} 
Let $H$ and $K$ be two subgroups of a group $G$. Then 
\[ssd(G) <  \frac{|G|^2}{|\mathcal{L}(G)|^2} \sum_{H,K \in \mathcal{L}(G)} d(H,K).\]
%If  ${\underset{K\in \mathcal{L}(G)} \bigcap}C_K(H)=1$, then  \[ssd(G)\le  \ \sum_{H,K \in %\mathcal{L}(G)} d(H,K)\]
%and the equality holds if and only if $H$ is cyclic. 
\end{thm}

\begin{proof} We claim that \begin{equation}\bigcup_{K\in \mathcal{L}(G)}C_K(H)=Comm_G(H).\end{equation} Let $T={\underset{K\in \mathcal{L}(G)} \bigcup} C_K(H)$ and $t \in T$. Then there exists a $K_t \in \mathcal{L}(G)$ containing $t$ such that $t\in C_{K_t}(H)$, that is, $[t,H]=1$, which means that $t$ permutes with all elements of $H$. In particular,  the powers of $t$ permutes with all elements of $H$ and so $[\langle t \rangle, H]=1$, which means $\langle t \rangle$ is in $Comm_G(H)$. We conclude that $T \subseteq Comm_G(H)$. Conversely, if $K \in \mathcal{L}(G)$ is in $Comm_G(H)$, then $[K, H]=1$ and so $K\subseteq C_G(H)$, then $K \subseteq T$. The claim follows.

%Now, denoting with $\dot{\cup}$ the disjoint union of two sets $A$ and $B$, we have obviously $A \cup B = (A\dot{\cup} B) \dot{\cup} (A \cap B)$, then    
%\[\left(\dot{\bigcup_{K\in \mathcal{L}(G)}}C_K(H)\right) \ \dot{\cup} \ \left({\bigcap_{K\in \mathcal{L}(G)}}C_K(H)\right)= \bigcup_{K\in \mathcal{L}(G)}C_K(H) \] and then

Therefore
\begin{equation}|\mathcal{L}(G)|^2 \ ssd(G)=  \sum_{H \in \mathcal{L}(G)}|Comm_G(H)|=\sum_{H \in \mathcal{L}(G)}\left|\bigcup_{K\in \mathcal{L}(G)}C_K(H)\right|\end{equation}
\[< \sum_{K \in \mathcal{L}(G)} \sum_{H \in \mathcal{L}(G)} |C_K(H)|\]
and we note that the equality cannot occur here as the identity $1 \in C_K(H)$ for all $H$ and $K$ in $\mathcal{L}(G)$. Since $C_K(H) \subseteq C_K(h)$ whenever $h \in H$, we may continue, finding the following upper bound
\begin{equation}\label{tech}
\le \sum_{K \in \mathcal{L}(G)} {\underset{H \in \mathcal{L}(G)} {\underset{h \in H}\sum}} |C_K(h)| = \sum_{H,K \in \mathcal{L}(G)} \Big( \sum_{h \in H}|C_K(h)|\Big) 
\end{equation}
\[=\sum_{H,K \in \mathcal{L}(G)} d(H,K) \ |H | \ |K| \le |G|^2 \ \sum_{H,K \in \mathcal{L}(G)} d(H,K).\]
\end{proof}

\begin{rem} We want just to illustrate two points of views which allow us to decide whether a group $G$ is abelian or not. The first deals with the subgroups: from Proposition \ref{trivial} $G$ is abelian if and only if $ssd(G)$ is trivial. The second deals with the elements: $G$ is abelian if and only if $d(G)$ is trivial. Theorem \ref{t:fundamental} is relevant, because it correlates $d(G)$ with $ssd(G)$. This is very helpful, because we have literature on $d(G)$ but few is known about $ssd(G)$ and $sd(G)$.
\end{rem}

In virtue of the previous remark, the following result is significative and answers partially some open questions in \cite{mtbis}. We will see, concretely, that the argument of Theorem \ref{t:fundamental} is very general and can be adapted to the context of $sd(G)$.

\begin{thm}\label{t:again} Let $H$ and $K$ be two subgroups of a group $G$. Then 
\[sd(G) \ge  \frac{1}{|\mathcal{L}(G)|^2} \sum_{H\in \mathcal{L}(G)} \left|\bigcap_{h \in H} C_K(h)\right| \] with
\[\sum_{H,K \in \mathcal{L}(G)} d(H,K) \ |H | \ |K| \ge \sum_{H,K\in \mathcal{L}(G)} \left|\bigcap_{h \in H} C_K(h)\right|.\]
\end{thm}

\begin{proof} From Theorem \ref{t:fundamental} (more precisely from \eqref{tech}), we may restrict to prove only  the first inequality. In order to do this, we claim that \begin{equation}C_K(H) \subseteq \bigcup_{K\in \mathcal{L}(G)}C_K(H) \subseteq \mathcal{C}(H).\end{equation} The first inclusion is trivial. Let $S={\underset{K\in \mathcal{L}(G)} \bigcup} C_K(H)$ and $s \in S$. Then there exists a $K_s \in \mathcal{L}(G)$ containing $s$ such that $s\in C_{K_s}(H)$, that is, $[s,H]=1$, which means that $s$ permutes with all elements of $H$. In particular,  $[\langle s \rangle, H]=1$ then $\langle s \rangle H = H \langle s \rangle$, which means $\langle s \rangle \in \mathcal{C}(H)$. We conclude that $S \subseteq \mathcal{C}(H)$.

%Now, denoting with $\dot{\cup}$ the disjoint union of two sets $A$ and $B$, we have obviously $A \cup B = (A\dot{\cup} B) \dot{\cup} (A \cap B)$, then    
%\[\left(\dot{\bigcup_{K\in \mathcal{L}(G)}}C_K(H)\right) \ \dot{\cup} \ \left({\bigcap_{K\in \mathcal{L}(G)}}C_K(H)\right)= \bigcup_{K\in \mathcal{L}(G)}C_K(H) \] and then

Therefore
\begin{equation}|\mathcal{L}(G)|^2 \ sd(G)=  \sum_{H \in \mathcal{L}(G)}|\mathcal{C}(H)| \ge \sum_{H \in \mathcal{L}(G)}\left|\bigcup_{K\in \mathcal{L}(G)}C_K(H)\right| 
\ge \sum_{H \in \mathcal{L}(G)}|C_K(H)| \end{equation}
but we observe that in general the following is true  
\begin{equation}{\underset{h \in H}\bigcap} C_K(h)= C_K(H)\end{equation}
so that
\begin{equation}
= \sum_{H \in \mathcal{L}(G)}\left|\bigcap_{h \in H} C_K(h)\right|.
% \ge \sum_{H \in \mathcal{L}(G)}\left|\bigcap_{h \in H} C_K(h)\right|
 \end{equation}
 On another hand, we note that 
 \begin{equation}
\sum_{H,K \in \mathcal{L}(G)} d(H,K) \ |H | \ |K|= \sum_{H,K \in \mathcal{L}(G)} \  \Big(\sum_{h \in H}|C_K(h)|\Big) \end{equation}
\[= \sum_{K \in \mathcal{L}(G)} \  \Big({\underset{H \in \mathcal{L}(G)}{\underset{h \in H}\sum}}|C_K(h)|\Big)  
\ge  \sum_{K \in \mathcal{L}(G)} \ \Big(\sum_{H \in \mathcal{L}(G)} \left|\bigcap_{h \in H} C_K(h)\right|\Big).\]

% \[ \sum_{H \in \mathcal{L}(G)}\Big(\sum_{h \in H}|C_K(h)|\Big)
%= \sum_{H,K \in \mathcal{L}(G)} \Big( \sum_{h \in H}|C_K(h)|\Big) \]
%\[=\sum_{H,K \in \mathcal{L}(G)} d(H,K) \ |H | \ |K|.\]
\end{proof}

In the rest of this section we reformulate $ssd(G)$ in terms of arithmetic functions. It is possible to rewrite $ssd(G)$ in the following form:
\begin{equation}ssd(G)= \frac{1}{|\mathcal{L}(G)|^2 } {\underset {X,Y \in \mathcal{L}(G)} \sum \varphi(X, Y)},
\end{equation} where  $\varphi :
\mathcal{L}(G)^2 \rightarrow \{0,1\}$ is the function defined by
\begin{equation}\varphi (X, Y)=\left\{ \begin{array}{lcl} 1,&\,\,& \mathrm{if} \ [X, Y]=1,\\
0,&\,\,& \mathrm{if} \ [X, Y]\not= 1 .\end{array} \right. \end{equation}
The reader may note that $\varphi(X,Y)=\varphi(Y,X)$, that is, $\varphi$ is symmetric in the variables $X$ and $Y$. There is a corresponding property of symmetry for the subgroup commutativity degree in \cite[Section 2]{mt}, but, in general, this property depends on the  permutability which we are going to study. For instance, this does not happen for weaker forms of permutability with respect to the maximal sugroups, as shown in \cite{or}. However, the introduction of the function $\varphi$ allows us to  simplify the notations. In fact, if $Z$ is a given subgroup of $G$ and we consider the sets $\mathcal{B}_1=\{(X \in  \mathcal{L}(G) :  Z \subseteq X \}$ and $\mathcal{B}_2=\{X \in \mathcal{L}(G) : X \subset Z\},$ then $\mathcal{B}_1 \cup
\mathcal{B}_2 \subseteq  \mathcal{L}(G)$ and so
\begin{equation}\label{e:1} |\mathcal{L}(G)|^2 \ ssd(G)\ge  \sum_{X, Y\in \mathcal{B}_1 \cup \mathcal{B}_2} \varphi(X,Y)\end{equation}
\[=\sum_{X,Y\in \mathcal{B}_1} \varphi(X,Y) + \sum_{X,Y\in \mathcal{B}_2} \varphi(X,Y)
+2 \sum_{X\in \mathcal{B}_1} \sum_{Y\in \mathcal{B}_2} \varphi(X,Y).\]

A consequence of this equation is examined below and overlaps a similar situation for $sd(G)$ in \cite{mt}.

\begin{prop}\label{p:1} Let $G$ be a group and $N$ be a normal subgroup of $G$. Then
\[ssd(G)\geq \frac{1}{|\mathcal{L}(G)|^2} \ \Big( \Big( |\mathcal{L}(N)|+
|\mathcal{L}(G/N)|-1\Big)^2\]
\[+(ssd(N)-1)|\mathcal{L}(N)|^2+(ssd(G/N)-1)|\mathcal{L}(G/N)|^2 \Big).\]
\end{prop}

\begin{proof} 
We are going to rewrite more properly the terms in the left side of \eqref{e:1}.
\begin{equation}
|\mathcal{L}(G/N)|^2 \ ssd(G/N)= {\underset{X,Y\in \mathcal{B}_1}\sum}\varphi(X,Y);
\end{equation}
\begin{equation}|\mathcal{L}(N)|^2 \ ssd(G/N) -2 |\mathcal{L}(N)| +1  = {\underset{X,Y\in \mathcal{B}_2 \cup
\{N\}}\sum}\varphi(X,Y)\end{equation}
\[-2 {\underset{X\in \mathcal{B}_2 \cup \{N\}}\sum}\varphi(X,N)+1={\underset{X,Y\in
\mathcal{B}_2}\sum}\varphi(X,Y);\]
\begin{equation}
2 |\mathcal{L}(G/N)| (|\mathcal{L}(N)|-1)  = 2 |\mathcal{B}_1| |\mathcal{B}_2| = 2 {\underset{X\in \mathcal{B}_1
}\sum} {\underset{Y\in \mathcal{B}_2}\sum}\varphi(X,Y).
\end{equation}
Replacing these expressions in \eqref{e:1}, the result follows.
\end{proof}

We list three consequences of Proposition \ref{p:1}, overlapping similar situations for $sd(G)$ in \cite{mt}. Their proof is omitted, since it is enough to note that for a normal abelian subgroup $N$ of $G$ we have $ssd(G/N)=1$ by Proposition \ref{p:1}, and, if it is of prime index in $G$, then $|\mathcal{L}(G/N)|=2$ .

\begin{cor}\label{c:1} Let $G$ be a group and $N$ be a normal subgroup of $G$ such that $G/N$ and $N$ are abelian.
Then \[ssd(G)\ge \frac{1}{|\mathcal{L}(G)|}\Big(|\mathcal{L}(N)|+ |\mathcal{L}(G/N)|-1\Big)^2.\]
\end{cor}

\begin{cor}\label{c:2} Let $G$ be a group and $N$ be a normal subgroup of $G$ of prime index.
Then \[ssd(G)\ge \frac{1}{|\mathcal{L}(G)|^2} \Big(ssd(N)|\mathcal{L}(N)|^2+2|\mathcal{L}(N)| +1\Big).\]
\end{cor}

%\begin{proof} 
%\end{proof}

%For solvable groups we may iterate the above argument, still following the ideas in %\cite{mt}.

\begin{cor}\label{c:3} A nonabelian solvable group $G$ has 
\[ssd(G)\ge \frac{1}{|\mathcal{L}(G)|^2} \Big(ssd(G')|\mathcal{L}(G')|^2+2|\mathcal{L}(G')| +1\Big).\] In particular, if $G$ is metabelian, then
\[ssd(G)\ge \frac{1}{|\mathcal{L}(G)|^2} \Big(|\mathcal{L}(G')|^2+2|\mathcal{L}(G')| +1\Big).\]
\end{cor}

%\begin{proof}
%Recall that the chief factors of a solvable group are of prime order. Therefore each %$G_i/G_{i-1}$ is of prime
%order for $i=1,2,\ldots,k$. If we apply Corollary \ref{c:2} to $G_i$, then  $|\mathcal{L}%(G)|^2 ssd(G_i)\ge
%(ssd(G_{i-1})|\mathcal{L}(G_{i-1})^2+2|\mathcal{L}(G_{i-1})| +1).$ The result follows, once %we sum up all these
%quantities.
%\end{proof}

Now we list some general bounds, related to subgroups and quotients. In a different context, these relations have been found in \cite{or}.

\begin{thm}\label{t:1}  Let $H$ be a subgroup of a group $G$. Then
\[\frac{|\mathcal{L}(H)|^2}{|\mathcal{L}(G)|^2} \ ssd(H)\leq  ssd(G) \]
and for all subgroups $L$ and $M$ of $H$ \[\frac{1}{|\mathcal{L}(G)|^2} \sum_{L\in \mathcal{L}(H)} \left|\bigcap_{l \in L} C_M(l)\right| \le   sd(H) \leq  sd(G).\]
\end{thm}

\begin{proof}  We proceed to prove the first inequality. The result is obviously true for $H=G$ and then we may assume $H\not=G$. Since $\mathcal{L}(H)\subseteq \mathcal{L}(G)$,
\begin{equation}|\mathcal{L}(H)|^2  \  ssd(H)={\underset{X,Y \in \mathcal{L}(H)}\sum} \varphi(X,Y)\le {\underset{X,Y \in
\mathcal{L}(G)}\sum} \varphi(X,Y)= |\mathcal{L}(G)|^2 \ ssd(G).\end{equation} The  inequality follows.

Now we proceed to prove the remaining part.  When we consider the corresponding function $\psi$, related to $sd(G)$ (details can be found in \cite{mt,mtbis}), instead of $\varphi$, we may overlap the previous argument and find that $\frac{|\mathcal{L}(H)|^2}{|\mathcal{L}(G)|^2}sd(H) \le sd(G)$. From Theorem \ref{t:again}, it follows that 
\begin{equation}\frac{1}{|\mathcal{L}(H)|^2} \sum_{L\in \mathcal{L}(H)} \left|\bigcap_{l \in L} C_M(l)\right|\le \ sd(H)\end{equation}
 then 
 \begin{equation} \frac{|\mathcal{L}(H)|^2}{|\mathcal{L}(G)|^2}  \ \Big(\frac{1}{|\mathcal{L}(H)|^2} \sum_{L\in \mathcal{L}(H)} \left|\bigcap_{l \in L} C_M(l)\right| \Big) \le  sd(H)
 \end{equation}
and the result follows.
\end{proof}

In \cite[Chapter 1]{rs}, it is shown that $\mathcal{L}(G_1 \times G_2) \ne \mathcal{L}(G_1) \times \mathcal{L}(G_2)$ in general, but if $G_1$ and $G_2$ have coprime orders then it is true. This motivates our assumption in the following proposition.

\begin{prop}\label{p:2}
For two groups $G_1$ and $G_2$ of coprime orders,
\[
ssd(G_1 \times G_2) = ssd(G_1)\cdot ssd(G_2).
\]
\end{prop}

\begin{proof}
We have $\mathcal{L}(G_1 \times G_2) = \mathcal{L}(G_1) \times \mathcal{L}(G_2)$, because $G_1$ and $G_2$ have coprime orders. Therefore, with obvious meaning of symbols, 
\begin{equation}ssd(G_1 \times G_2)  = \frac{1}{|\mathcal{L}(G_1 \times G_2)|^2}\underset{A_1 \times A_2 \in \mathcal{L}(G_1 \times G_2) }{\sum} |Comm_{G_1 \times G_2}(A_1 \times A_2)|\end{equation}
\[ = \frac{1}{|\mathcal{L}(G_1)\times \mathcal{L}(G_2)|^2}\underset{A_1 \times A_2 \in \mathcal{L}(G_1) \times \mathcal{L}(G_2) }{\sum} |Comm_{G_1}(A_1) \times Comm_{G_2}(A_2)|\]
\[ = \left(\frac{1}{|\mathcal{L}(G_1)|^2}\underset{A_1 \in \mathcal{L}(G_1) }{\sum} |Comm_{G_1}(A_1)|\right) \left(\frac{1}{|\mathcal{L}(G_2)|^2}\underset{A_2 \in \mathcal{L}(G_2) }{\sum} |Comm_{G_2}(A_2)|\right)\]
\[ = ssd(G_1)\cdot ssd(G_2).\]
Hence the proposition follows.
\end{proof}

\begin{cor}\label{c:4}
Proposition \ref{p:2} is still true for finitely many factors.
\end{cor}
\begin{proof}
We can mimick the proof of Proposition \ref{p:2}.
\end{proof}

\section{Multiple strong subgroup commutativity degree}
In analogy with $d^{(n)}(H, G)$ ($n \geq 1$), introduced in \cite{err}, the notion of strong subgroup commutativity degree, given in Section 1, can be further generalized in the following way:
\begin{equation}
ssd^{(n)}(H,G) = \frac{|\{(L_1,\dots, L_n, K) \in \mathcal{L}(H)^n \times \mathcal{L}(G)  \ | \  [L_1,\dots, L_n, K] = 1\}|}{|\mathcal{L}(H)|^n \ |\mathcal{L}(G)|}. 
\end{equation}
In particular, if $n=1$ and $H=G$, then $ssd^{(1)}(G, G)=ssd(G)$. Briefly, $ssd^{(n)}(H)$ denotes
 \begin{equation}ssd^{(n)}(H,H)=\frac{|\{(L_1,\dots, L_n, L_{n+1}) \in \mathcal{L}(H)^{n+1}   \ | \  [L_1,\dots, L_n, L_{n+1}] = 1\}|}{|\mathcal{L}(H)|^{n+1}}.\end{equation} On another hand, we note that \begin{equation}[L_1,\dots, L_n, K] = [[L_1,\dots, L_n],K] = \ldots =[[ \ldots [[L_1,L_2], L_3] \ldots L_n], K]= 1\end{equation} and so
\begin{equation}Comm_G(L_1,\dots, L_n) = \{K \in \mathcal{L}(G) \ | \  [L_1,\dots, L_n,K] = 1\},\end{equation}
\begin{equation}Comm_{H\times G}(L_1,\dots, L_{n-1}) = \{(L_n, K) \in \mathcal{L}(H) \times \mathcal{L}(G) \ | \  [[[L_1,\dots,L_{n-1}], L_n],K] = 1\}\end{equation}
\[\ldots \ldots \ldots \ldots \ldots \ldots \ldots \ldots \ldots \ldots  \ldots \ldots \ldots \ldots \ldots \ldots\]
$Comm_{H^{n-1}\times G}(L_1) = \{(L_2,L_3,\ldots,L_n, K) \in \mathcal{L}(H)^{n-1} \times \mathcal{L}(G) \ |$  \newline $ \  [\ldots [L_1,L_2],\ldots,  L_n],K] = 1\}.$

Of course, all these sets are nonempty, since they contain at least the trivial subgroup. By construction, $Comm_{H^{n-1}\times G}(L_1) \subseteq Comm_{H^{n-2}\times G}(L_1,L_2) \subseteq \ldots \subseteq Comm_{H\times G}(L_1,\dots, L_{n-1}) \subseteq Comm_G(L_1,\dots, L_n).$ From the above inclusions we observe that for $n$ which is growing the $Comm_{H^{n-1}\times G}(L_1)$ is getting to the trivial subgroup.
Therefore
\begin{equation}|\mathcal{L}(H)|^n \ |\mathcal{L}(G)| \ ssd^{(n)}(H,G) =  {\underset{L_1,\dots, L_n \in \mathcal{L}(H)} \sum} |Comm_G(L_1,\dots, L_n)|\end{equation}
\[={\underset{L_1,\dots, L_n \in \mathcal{L}(H)} \sum} |Comm_{H^{n-1}\times G}(L_1)| \]
 and to the extreme case we have
 \begin{equation}\lim_{n\rightarrow \infty} \ ssd^{(n)}(H,G)= \lim_{n\rightarrow \infty} \frac{1}{|\mathcal{L}(H)|^n \ |\mathcal{L}(G)|} \ \cdot \ \lim_{n\rightarrow \infty}  {\underset{L_1,\dots, L_n \in \mathcal{L}(H)} \sum} |Comm_{H^{n-1}\times G}(H_1)|\end{equation}
 \[= \frac{1}{|\mathcal{L}(G)|} \ \cdot \ \lim_{n\rightarrow \infty} \frac{1}{|\mathcal{L}(H)|^n} \ \cdot \ 1 =0.\]
This is a qualitative argument which shows that it is meaningful to consider values of probabilities of $ssd^{(n)}(H,G)$ for a small number of commuting subgroups. At the same time, the above construction shows that $ssd^{(n)}(H,G)$ is a strictly decreasing sequence of numbers in $[0,1]$ in the variable  $n$. Namely,
\begin{equation}ssd^{(1)}(H,G)\ge ssd^{(2)}(H,G) \ge \ldots  \ge ssd^{(n)}(H,G) \ge ssd^{(n+1)}(H,G) \ge \ldots  \end{equation}

We want to point out that  a similar treatment can be done for $sd(G)$, as proposed in a series of opens problems in \cite{mtbis}, where the corresponding version of $ssd^{(n)}(H,G)$ is called \textit{relative subgroup commutativity degree}. 
%In particular, it is easy to see that 
%\begin{equation}\lim_{n\rightarrow \infty} \ sd^{(n)}(H,G)=0
%\end{equation}
%and
%\begin{equation}sd^{(1)}(H,G)\ge sd^{(2)}(H,G) \ge \ldots  \ge sd^{(n)}(H,G) \ge sd^{(n+1)}%(H,G) \ge \ldots  \end{equation}

As done in Section 2, we may rewrite $ssd^{(n)}(H,G)$ in the following form:
\begin{equation}ssd^{(n)}(H,G)= \frac{1}{|\mathcal{L}(H)|^n \ |\mathcal{L}(G)| } {\underset {\underset {Y\in \mathcal{L}(G)} {X_1,\ldots, X_n \in \mathcal{L}(H)}} \sum \varphi_n(X_1,\ldots, X_n, Y)},
\end{equation} where  $\varphi_n :
\mathcal{L}(H)^n \times \mathcal{L}(G) \rightarrow \{0,1\}$ is the function defined by
\begin{equation}\varphi_n (X_1,\ldots, X_n, Y)=\left\{ \begin{array}{lcl} 1,&\,\,& \mathrm{if} \ [X_1, \ldots, X_n, Y]=1,\\
0,&\,\,& \mathrm{if} \ [X_1,\ldots, X_n, Y]\not= 1 \end{array} \right. \end{equation}
and continues to be symmetric.

\begin{prop} \label{boundingsdg}
Given  subgroup $H$ of a group $G$, \[ssd^{(n)}(H,G) \le ssd^{(n)}(G,G) \le ssd(G) \leq sd(G).\]
\end{prop}
\begin{proof}
We begin to prove  the first inequality. Since $\mathcal{L}(H) \subseteq \mathcal{L}(G)$,
\begin{equation}ssd^{(n)}(H,G) \le |\mathcal{L}(H)|^n  \ |\mathcal{L}(G)| \ ssd^{(n)}(H, G) =  {\underset {\underset {Y\in \mathcal{L}(G)} {X_1,\ldots, X_n \in \mathcal{L}(H)}} \sum \varphi_n(X_1,\ldots, X_n, Y)}\end{equation}
\begin{equation}\le \sum_{X_1,\ldots, X_n,Y \in \mathcal{L}(G)} \varphi_n(X_1,\ldots, X_n, Y)
= |\mathcal{L}(G)|^n  \ |\mathcal{L}(G)| \ ssd^{(n)}(G, G). \end{equation}
The second inequality follows once we note that $ssd^{(n)}(H,G)$  is a decreasing sequence. Therefore, if we fix $H=G$, then $ssd(G)=ssd^{(1)}(G,G)\ge ssd^{(2)}(G,G)\ge \ldots \ge ssd^{(n)}(G,G)\ge \ldots $.
The last inequality follows once we note that $Comm_G(H) \subseteq \mathcal{C}(H)$ and that
\begin{equation} ssd(G) = \frac{1}{|\mathcal{L}(G)|^2}  \sum_{H\in \mathcal{L}(G)} |Comm_G(H)| \leq \frac{1}{|\mathcal{L}(G)|^2}\sum_{ H\in \mathcal{L}(G)}|\mathcal{C}(H)| = sd(G). \end{equation}
\end{proof}

\begin{prop}\label{p:3}
For two groups $C$ and $D$ of coprime orders and two subgroups $A\le C$ and $B\le D$,
\[
ssd^{(n)}(A \times B, C \times D) = ssd^{(n)}(A,C)\cdot ssd^{(n)}(B,D).
\]

\end{prop}
\begin{proof}
\begin{equation}ssd^{(n)}(A\times B, C\times D)\end{equation}
\[ = \frac{1}{|\mathcal{L}(A\times B)|^n \ |\mathcal{L}(C\times D)|} \underset{A_1\times B_1,\ldots, A_n\times B_n \in \mathcal{L}(A\times B)}{\sum} |Comm_{A \times B}(A_1 \times B_1,\dots, A_n \times B_n)|\]
\[=\frac{1}{|\mathcal{L}(A)|^n \cdot |\mathcal{L}(B)|^n \cdot |\mathcal{L}(C)| \cdot |\mathcal{L}(D)|} \Big(\underset{A_1\times B_1,\ldots, A_n\times B_n \in \mathcal{L}(A\times B)}{\sum} |Comm_A(A_1,\ldots, A_n)|\]
\[ \cdot |Comm_B(B_1,\ldots, B_n)|\Big) =\frac{1}{|\mathcal{L}(A)|^n \cdot |\mathcal{L}(B)|^n \cdot |\mathcal{L}(C)| \cdot \mathcal{L}(D)|}\]
\[= \Big(\underset{A_1,\ldots, An \in \mathcal{L}(A)}{\sum} |Comm_A(A_1,\dots, A_n)|\Big) \cdot \Big(\underset{B_1,\ldots, B_n \in \mathcal{L}(B)}{\sum} |Comm_B(B_1,\ldots, B_n)|\Big)\]
\[=\frac{1}{|\mathcal{L}(A)|^n \ |\mathcal{L}(C)| } \Big(\underset{A_1,\ldots, An \in \mathcal{L}(A)}{\sum} |Comm_A(A_1,\dots, A_n)|\Big)\]
\[ \cdot \frac{1}{ |\mathcal{L}(B)|^n \ |\mathcal{L}(D)|}\Big(\underset{B_1,\ldots, B_n \in \mathcal{L}(B)}{\sum} |Comm_B(B_1,\ldots, B_n)|\Big)\]
\[= ssd^{(n)}(A,C)\cdot ssd^{(n)}(B,D).\]
\end{proof}

We note that Proposition \ref{p:2} follows from Proposition \ref{p:3}, when $n=1$, $A=C=G_1$, $B=D=G_2$.

\begin{cor}\label{c:5}
Proposition \ref{p:3} is still true for finitely many factors.
\end{cor}
\begin{proof}
We can mimick the proof of Proposition \ref{p:3}.
\end{proof}

We end with a variation on the theme of Theorems \ref{t:fundamental} and \ref{t:again}.

\begin{thm} 
Let $H$ and $K$ be two subgroups of a group $G$. Then for all $n \ge1$ 
\[ssd^{(n)}(H,H) < \frac{|H|^{n+1}}{|\mathcal{L}(H)|^{n+1}} \ \sum_{K \in \mathcal{L}(H)}  d^{(n)}(K,K).\] 
%If  $\bigcap_{K \in \mathcal{L}(H)} C_K(H^n)=1$, then  \[ssd^{(n)}(H,H)\le  \ \sum_{K \in \mathcal{L}(H)} d^{(n)}(K,K)\] and the equality holds if and only if $H^n$ is direct product of $n$ cyclic groups. 
\end{thm}

\begin{proof} 
Overlapping the argument in the proof of Theorem \ref{t:fundamental},we firstly prove that \begin{equation}\bigcup_{(L_2,\ldots, L_n, L_{n+1})\in \mathcal{L}(H)^n}C_{H^n}(L_1)=Comm_{H^n}(L_1),\end{equation} 
where \begin{equation}
Comm_{H^n}(L_1)=Comm_{H^{n-1}\times H}(L_1)\end{equation}
\[ = \{(L_2,L_3,\ldots, L_n, L_{n+1}) \in \mathcal{L}(H)^{n-1} \times \mathcal{L}(H) \ |   \  [\ldots [L_1,L_2],\ldots,  L_n],L_{n+1}] = 1\}\]
and then 
\begin{equation}|\mathcal{L}(H)|^{n+1} \ ssd^{(n)}(H,H)=  \sum_{L_1\in \mathcal{L}(H)} |Comm_{H^n}(L_1)|\end{equation}		
\[=\sum_{L_1 \in \mathcal{L}(H)} \left| \bigcup_{(L_2,\ldots , L_n, L_{n+1})\in \mathcal{L}(H)^n} C_{H^n}(L_1)\right| \]
\[< \sum_{(L_2,\ldots, L_n, L_{n+1})\in \mathcal{L}(H)^n} \sum_{L_1 \in \mathcal{L}(H)^n} |C_{H^n}(L_1)|\]
and we note that the equality must be strict for the same motivation of the corresponding step in the proof of Theorem \ref{t:fundamental}. Since $C_{H^n}(L_1) \subseteq C_{H^n}(l_1)$ whenever $l_1 \in L_1$, we may continue, finding that
\begin{equation}\label{tech}
\le \sum_{(L_2,\ldots, L_n, L_{n+1}) \in \mathcal{L}(H)^n} {\underset{L_1 \in \mathcal{L}(H)} {\underset{l_1 \in L_1}\sum}} |C_{H^n} (l_1)|\end{equation}
\[ = \sum_{(L_1,L_2,\ldots, L_n,L_{n+1}) \in \mathcal{L}(H)^{n+1}} \Big( \sum_{l_1 \in L_1}|C_{H^n}(l_1)|\Big) \]
\[=\sum_{K \in \mathcal{L}(H)} d^{(n)}(K,K) \ |K|^{n+1} \le |H|^{n+1} \ \sum_{K \in \mathcal{L}(H)} d^{(n)}(K,K).\]
\end{proof}

Roughly speaking, in the proof of Theorem \ref{t:1} we may replace the role  of $\varphi=\varphi_2$ with that of $\varphi_n$ for $n>2$. We will find the following generalization of Theorem \ref{t:1}, whose proof is easy to check and so it is omitted.

\begin{thm}\label{t:2}  Let $H$ be a subgroup of a group $G$. Then for all $n\ge1$
%\begin{itemize}
%\item [(i)]   
\[\frac{|\mathcal{L}(H)|^{n+1}}{|\mathcal{L}(G)|^{n+1}} \ ssd^{(n)}(H)\leq   ssd^{(n)}(G).\]
%\item [(ii)]  $ssd^{(n)}(G)^2\leq ssd^{(n)}(G/N) \ ssd^{(n)}(N)$.
%\end{itemize}
\end{thm}

We note that a similar treatment can be done for the relative subgroup commutativity degree in \cite{mtbis}, since the arguments involve only combinatorial properties and set theory. This fact motivates to conjecture that the context of infinite compact groups, once a suitable Haar measure is replaced with $ssd(G)$ or with $sd(G)$, may be subject to an analogous treatment.

\section{Two applications} 
Here we illustrate an application to the theory of characters and another to the dihedral
groups. Relations with the theory of characters are due to the fact that in a group $G$ 
\begin{equation}d(G)=\frac{|\mathrm{Irr}(G)|}{|G|},\end{equation}
where $\mathrm{Irr}(G)$ denotes the set of all irreducible complex characters of $G$. For  an element $g$ of $G$, let  
\begin{equation}\xi (g)=| (X, Y) \in \mathcal{L}(\langle g \rangle)\times \mathcal{L}(G) \ | \ [X, Y] = 1\}|.\end{equation} Thus,
\begin{equation}
ssd(\langle g \rangle, G) = \frac{\xi (g)}{|\mathcal{L}(\langle g \rangle)||\mathcal{L}(G)|}.
\end{equation}

\begin{lem}\label{cf}
$\xi(g)$ is a class function.
\end{lem}
\begin{proof}
It is enough to note that, for each $a \in G$, the map \begin{equation}f : (X, Y) \mapsto f(X,Y)=(aXa^{-1}, aYa^{-1})\end{equation} defines a one to one correspondence between the sets $\{(X, Y) \in \mathcal{L}(\langle g \rangle) \times \mathcal{L}(G)  \ | \ [X, Y] = 1\}$
and $\{(X, Y) \in \mathcal{L}(\langle aga^{-1} \rangle) \times \mathcal{L}(G)  \ | \ [X, Y] = 1\}$.
\end{proof}
Thus, it is meaningful to write
\begin{equation}
\xi(g) = \underset{\chi \in \mathrm{Irr}(G)}{\sum} [\xi, \chi] \chi(g)
\end{equation}
where $[\,,\, ]$ denotes the usual inner product of characters, defined by
\begin{equation}
[\xi, \chi] = \dfrac{1}{|G|}\sum_{g\in G}\xi(g)\overline{\chi(g)}=\dfrac{1}{|G|}\sum_{g\in G}\xi(g)\chi(g^{-1}).
\end{equation}

We recall that a class function defined on a finite group $G$ is said to be an $R$--\textit{generalized character of} $G$, for any ring $\mathbb{Z} \subseteq R \subseteq \mathbb{C}$, if it is an $R$--linear combination of irreducible complex characters of $G$. 

\begin{thm} \label{genchar}
$\xi$ is a $\mathbb{Q}$-generalized character of $G$. 
\end{thm}

\begin{proof}
Clearly, if two elements $x$ and $y$ of $G$ generate the same cyclic group then $\xi(x) = \xi(y)$. Suppose that $o(x)= o(y) = n$. Let $\varepsilon$ be a primitive $n$th root of unity. We have $y = x^m$ 
for some $m$ with $(m, n)=1$ and thus $\varepsilon^m$ is a primitive $n$th root of unity. As usual, $ \mathrm{Gal}(\mathbb{Q}[\varepsilon]/\mathbb{Q})$ denotes the Galois group, related to the algebraic extension $\mathbb{Q}[\varepsilon]$ over $\mathbb{Q}$, obtained adding $\varepsilon$.   Therefore, for any $\sigma \in \mathrm{Gal}(\mathbb{Q}[\varepsilon]/\mathbb{Q})$ we have \begin{equation}\chi(x)^{\sigma} = \sum {\epsilon_i}^\sigma = \sum {\epsilon_i}^m  =  \chi(x^m).\end{equation}   Thus for any $\chi\in \mathrm{Irr}(G)$ and $g\in G$, 
\begin{equation}
\chi(g)^{\sigma} =    \chi(g^m)
\end{equation}
and hence $\left(\delta(g)\chi(g^{-1})\right)^{\sigma}=\delta(g^m)\chi(g^{-m})$. Hence $\sigma$ fixes $\sum_{g\in G}\delta(g)\chi(g^{-1})$ and this completes the proof.
\end{proof}

\begin{cor}
$|G|[\xi, \chi]$ is an integer for all $\chi\in \mathrm{Irr}(G)$.
\end{cor}
\begin{proof}
Since $\chi(g)$ is an algebraic integer the result follows from Lemma  \ref{cf} and Theorem \ref{genchar}.   
\end{proof}

For the second application, the dihedral group
\begin{equation}D_{2n} = \langle x, y \ | \ x^2 = y^n = 1, x^{-1}yx=y^{-1} \rangle \end{equation}
of symmetries of a regular polygon with $n \ge 1$ edges has order $2n$ and a well--known de-
scription of $|\mathcal{L}(D_{2n})|$ can be found in \cite{rs,mt,mtbis}. For instance, it is easy to see that $D_{2n} \simeq C_2\ltimes C_n$ is the semidirect product of a cyclic group $C_2$ of order 2 acting by inversion on a cyclic group $C_n$ of order $n$. For every divisor $r$ of $n$, $D_{2n}$ has a subgroup isomorphic to $C_r$ , namely $\langle x^{\frac{n}{r}}\rangle$, and $\frac{n}{r}$ subgroups isomorphic to $D_{2r}$, namely
$\langle x^\frac{n}{r} , x^{i-1} , y\rangle$ for $i = 1, 2, \ldots, \frac{n}{r}$. Then
\begin{equation}
|\mathcal{L}(D_{2n})| = \sigma(n) + \tau (n),\end{equation}
where $\sigma(n)$ and $\tau (n)$ are the sum and the number of all divisors of $n$, respectively. The next result generalizes the above considerations, when we have a group with a structure very close to that of $D_{2n}$.

\begin{cor}\label{appl}Assume that $G$ is a metabelian group of even order. If $|\mathcal{L}(G)| = \sigma(\frac{|G|}{2}) + \tau (\frac{|G|}{2})$ and $G'$ is cyclic, then
\[\frac{(\tau(G')+1)^2}{ \Big(\sigma\Big(\frac{|G|}{2}\Big) + \tau \Big(\frac{|G|}{2}\Big)\Big)^2} \le \sum_{H,K \in \mathcal{L}(G)}\varphi(H,K) \le \frac{|G|^2}{ \Big(\sigma\Big(\frac{|G|}{2}\Big) + \tau \Big(\frac{|G|}{2}\Big)\Big)^2} \sum_{H,K \in \mathcal{L}(G)}d(H,K).\]
\end{cor}

\begin{proof}
Since $G'$ is cyclic, $|\mathcal{L}(G')| = \tau(G')$. Then the lower bound follows from Corollary \ref{c:3}, specifying the numerical values of the subgroup lattices. From Theorem \ref{t:fundamental}, we get  the upper bound, adapted to our case. The result follows.
\end{proof}

Corollary \ref{appl} is a counting formula for the number of permuting subgroups via $\varphi$, or, equivalently, via the strong subgroup commutativity degree and the commutativity degree. This observation is important in virtue of the fact that we know explicitly $d(H,K)$ by results
in \cite{ar, das-nath1, das-nath2, elr, err, lescot1, lescot2}.

%\section*{Acknowledgement}
%I am grateful to M. Farrokhi, who informed me about \cite{f1,f2,f3}, to A. Erfanian for %helpful discussions in a meeting in February 2011, to R.K. Nath for constructive comments in %the original version of the present manuscript, (especially in  Section 4), finally to A. %Alghamdi for answering my questions on the related topics of character theory.


\begin{thebibliography}{20}
\bibitem{aor}A.M. Alghamdi, D.E. Otera and F.G. Russo, On some recent investigations of probability in group theory, \textit{Boll. Mat. Pura Appl.} \textbf{3} (2010), 87--96.


\bibitem{ar}A.M. Alghamdi and F.G. Russo, \textit{A generalization of the probability that the commutator of two group elements is equal to a given element}, preprint,Cornell University Library, 2010, arXiv: 1004.0934.


\bibitem{barman} R. Barman, \textit{Quasinormality degrees of subgroups of a finite group and a class function}, preprint, 2011.

\bibitem{bM06}
F. Barry, D. MacHale and  \'A. N\'i Sh\'e, Some supersolvability conditions for finite groups, \textit{Math. Proc. Royal Irish  Acad.} \textbf{106 A} (2) (2006),   163--177.


\bibitem{b1}Y. Berkovich, \textit{Groups of prime power order Vol. I}, de Gruyter, Berlin, 2008.


\bibitem{cms} K. Chiti, M.R.R. Moghaddam and A. Salemkar, $n$-isoclinism classes and $n$-nilpotency degree of finite groups, \textit{Algebra Colloq.} \textbf{12} (2005), 255--261.


%\bibitem{b2}Y. Berkovich and Z. Janko, \textit{Groups of prime power order Vol. II}, de Gruyter, Berlin, 2008.


\bibitem{das-nath1} A.K. Das and R.K. Nath, On the generalized relative commutative degree of a finite group, \textit{Int. Electr. J. Algebra} \textbf{7} (2010), 140--151.

\bibitem{das-nath2} A.K. Das and R.K. Nath, On a lower bound of commutativity degree, \textit{Rend. Circ. Mat. Palermo} \textbf{59} (2010), 137--142.

%\bibitem{das-nath3} A.K. Das and R.K. Nath, On generalized commutativite degree of a finite %group, \textit{Rocky Mountain J. Math.}, to appear.

\bibitem{elr}  A. Erfanian,  P. Lescot and  R. Rezaei, On the relative commutativity degree of a subgroup of a finite group, \textit{Comm. Algebra} {\bf{35}} (2007), 4183--4197.


\bibitem{er} A. Erfanian  and  F.G. Russo, Probability of mutually commuting $n$-tuples in some classes of compact groups, \textit{Bull. Iran. Math. Soc.} \textbf{34} (2008), 27--37.



\bibitem{erfanian-rezaei} A. Erfanian and R. Rezaei, On the commutativity degree of compact groups, \textit{Arch. Math. (Basel)} \textbf{93} (2009), 345--356.

\bibitem{err} A. Erfanian, R. Rezaei and F.G. Russo, \textit{Relative $n$-isoclinism classes and relative $n$-th nilpotency degree of finite groups}, preprint, Cornell University Library, 2010, arXiv: 1003.2306.


\bibitem{f1} M. Farrokhi, H. Jafari and F. Saeedi, Subgroup normality degree of finite 
groups I, \textit{Arch. Math. (Basel)} \textbf{96} (2011), 215--224.

\bibitem{f2} M. Farrokhi and F. Saeedi, \textit{Subgroup normality degree of finite groups II}, preprint, 2011.

\bibitem{f3} M. Farrokhi, \textit{Finite groups with two subgroup normality degrees}, preprint, 2011.


\bibitem{guralnick} R.M. Guralnick and G.R. Robinson,  On the commuting probability in finite groups, \textit{J. Algebra} \textbf{300} (2006), 509--528.


\bibitem{hr} K.H. Hofmann and F.G. Russo, \textit{The probability that $x$ and $y$ commute in  a compact group}, preprint, Cornell University Library, 2010, arXiv:1001.4856.


\bibitem{lescot1}P. Lescot, Isoclinism classes and commutativity degrees of finite groups, \textit{J. Algebra} \textbf{177} (1995), 847--869.

\bibitem{lescot2}P. Lescot, Central extensions and commutativity degree, \textit{Comm. Algebra} \textbf{29} (2001), 4451--4460.

\bibitem{salemkar}
H. Mohammadzadeh,  A. Salemkar and H. Tavallaee, A remark on the commuting probability in finite groups, \textit{Southeast Asian Bull. Math.} \textbf{34} (2010), 755--763,


\bibitem{nr1}P. Niroomand and R. Rezaei, On the exterior degree of finite groups, \textit{Comm. Algebra} \textbf{39} (2011), 335--343. 

\bibitem{nr2}P. Niroomand and R. Rezaei,  The exterior degree of a pair of finite groups, preprint, Cornell University Library, arXiv:1101.4312v1, 2011.

\bibitem{nrr}P. Niroomand,  R. Rezaei and F.G. Russo,  Commuting powers and exterior degree of finite groups, preprint, Cornell University Library, arXiv:1102.2304, 2011.


\bibitem{or}
D.E. Otera and F.G. Russo, {\it Subgroup $S$-commutativity degree of finite groups}, preprint, Cornell University Library, 2010, arXiv:1009.2171.


%\bibitem{rezaei-russo2}R. Rezaei and F.G. Russo, Bounds for the relative $n$-th nilpotency %degree in compact groups, \textit{Asian-European J. Math.}, to appear.

\bibitem{rezaei-russo3} R. Rezaei and F.G. Russo,   $n$-th relative nilpotency degree and relative $n$-isoclinism classes, \textit{Carpathian J. Math.} \textbf{27} (2011), 123--130.


\bibitem{rusin}
D.J. Rusin,  What is the probability that two elements of a finite group commute?, \textit{Pacific J. Math.} {\bf 82} (1979), 237--247.

\bibitem{russo1} F.G. Russo, A probabilistic meaning of certain quasinormal subgroups, \textit{Int. J. Algebra} \textbf{1} (2007), 385--392.


\bibitem{russo2}   F.G. Russo, The generalized commutativity degree in a finite group, \textit{Acta Univ. Apulensis Math. Inform.} \textbf{18} (2009), 161--167.



\bibitem{rs}   R. Schmidt, \textit{Subgroup Lattices of Groups}, de Gruyter, Berlin, 1994.

\bibitem{mt}   M. T\v{a}rn\v{a}uceanu, Subgroup commutativity degrees of finite groups, \textit{J. Algebra} \textbf{321} (2009), 2508--2520.

\bibitem{mtbis}   M. T\v{a}rn\v{a}uceanu, Addendum [Subgroup commutativity degrees of finite groups], \textit{J. Algebra} \textbf{} (2011), in press.

\end{thebibliography}
\end{document}